\documentclass[reqno]{amsart}
\usepackage{amsmath, amssymb, amsthm, epsfig}
\usepackage{hyperref, latexsym}
\usepackage{url}
\usepackage[mathscr]{euscript}
\usepackage{mathabx}
\usepackage{color}
\usepackage{fullpage} 
\usepackage{setspace}

\allowdisplaybreaks

\def\today{\ifcase\month\or
  January\or February\or March\or April\or May\or June\or
  July\or August\or September\or October\or November\or December\fi
  \space\number\day, \number\year}

 \newtheorem{theorem}{Theorem}
  
 \newtheorem{lemma}[theorem]{Lemma}
 \newtheorem{proposition}[theorem]{Proposition}
 \newtheorem{corollary}[theorem]{Corollary}
 \theoremstyle{definition}

 \theoremstyle{remark}

 \newcommand{\C}{\mathbb{C}}
 \newcommand{\R}{\mathbb{R}}
 \newcommand{\N}{\mathbb{N}}
 
 \newcommand{\Z}{\mathbb{Z}}

 \newcommand{\hh}{\tfrac12}

  \renewcommand{\d}{\text{\rm d}}

\newcommand{\re}{{\rm Re}\,}

\begin{document}
\title[A note on the mean values of the derivatives of  $\zeta'/\zeta$]{A note on the mean values of\\ the derivatives of  $\zeta'/\zeta$}
\author[Chirre]{Andr\'{e}s Chirre}

\address{Department of Mathematical Sciences, Norwegian University of Science and Technology, NO-7491 Trondheim, Norway}

\email{carlos.a.c.chavez@ntnu.no }

\allowdisplaybreaks
\numberwithin{equation}{section}

\maketitle  

\begin{abstract}
Assuming the Riemann hypothesis, we obtain a formula for the mean value of the $k$-derivative of  $\zeta'/\zeta$, depending on the pair correlation of zeros of the Riemann zeta-function. This formula allows us to obtain new equivalences to Montgomery's pair correlation conjecture. This extends a result of Goldston, Gonek, and Montgomery where the mean value of $\zeta'/\zeta$ was considered.
\end{abstract}


\section{Introduction}
Let $\zeta(s)$ denote the Riemann zeta-function. The Riemann hypothesis (RH) states that the non-trivial zeros $\rho$ of $\zeta(s)$ have the form $\rho=\hh+i\gamma$ with $\gamma\in\R$.  We will assume RH throughout this paper.

\subsection{Montgomery's pair correlation conjecture} 
In $1973$, Montgomery \cite{M} defined the pair correlation function 
$$
	\displaystyle N(\beta,T):=\!\!\!\sum_{\substack{ 0<\gamma,\gamma'\le T \\ 0<\gamma-\gamma' \le \frac{2\pi \beta}{\log T} }} \!1,
$$
where the double sum runs over the ordinates $\gamma,\gamma'$ of two sets of non-trivial zeros of $\zeta(s)$, counted with multiplicity. Since there are $\displaystyle \sim T\log T/(2\pi)$ non-trivial zeros of $\zeta(s)$ with ordinates in the interval $(0,T]$ as $T\to \infty$, the function $N(\beta,T)$ counts the number of pairs of zeros within $\beta$ times the average spacing between zeros. The pair correlation conjecture of Montgomery asserts that
\begin{align} \label{PCC}
	\displaystyle N(\beta,T) \sim  \frac{T \log T}{2\pi}  \int_0^\beta \left\{ 1 - \Big( \frac{\sin \pi u}{\pi u}\Big)^{\!2} \right\} \mathrm{d}u, \,\,\,\, \text{as } T\to \infty \,\,\text{ for any fixed} \   \beta>0.
\end{align} 
Assuming RH, there are several known equivalences\footnote{\,\,\,\,\,\,For an equivalence of the pair correlation conjecture related to the asymptotic formula for an integral of Selberg connected with primes in short intervals, see \cite{GaMu, G, GGM}.} to this conjecture. Define the function
\[
F(\alpha,T) := \frac{2\pi}{T\log T} \sum_{0<\gamma,\gamma'\le T} T^{i \alpha (\gamma-\gamma')} w(\gamma-\gamma'),
\] 
introduced by Montgomery \cite{M}, where $\alpha \in \mathbb{R}$, $T\geq2$, and $w(u)=4/(4+u^2)$. Using this function, Goldston \cite{G} showed that the pair correlation conjecture \eqref{PCC} is equivalent to 
\begin{equation}\label{22_09pm}
	\displaystyle \int_{b}^{b+\ell}\!\!\! F(\alpha,T) \,\d\alpha  \sim  \ell, \,\,\,\, \text{as } T\to \infty \, \,\text{ for any fixed} \  b\ge 1 \ {\rm and} \ \ell > 0.
\end{equation}
Another equivalence for the pair correlation conjecture is related to the second moment of $\zeta'/\zeta$. In fact, Goldston, Gonek, and Montgomery  \cite[Theorem 3]{GGM} established that the pair correlation conjecture is  equivalent to the asymptotic
\begin{equation} \label{22_32pm}
	I(a,T) := \int_1^T \left| \frac{\zeta'}{\zeta}\!\left( \frac{1}{2}+\frac{a}{\log T}+it\right) \right|^2  \!\!\mathrm{d}t \sim  \left( \frac{1-e^{-2a}}{4 a^2} \right) T \log^2 T, \,\,\,\, \text{as } T\to \infty \,\,\text{ for any fixed } a>0.
\end{equation}

\smallskip

Since Montgomery's pair correlation conjecture remains a difficult open problem, the efforts have thus been concentrated in obtaining upper and lower bounds for the functions $N(\beta,T)$, $\int_{b}^{b+\ell}F(\alpha,T)\,\d\alpha$, and $I(a,T)$ in place of asymptotic formulae (see for instance \cite{CCLM, CCCM, Ga, G2, GG, GGM}).

\smallskip

\subsection{Mean values of the $k$-derivative of  $\zeta'/\zeta$} The main goal in this paper is to extend the technique developed by Goldston, Gonek, and Montgomery in \cite{GGM} to get new equivalences of the pair correlation conjecture, related to the mean values of the derivatives of $\zeta'/\zeta$. Let $k\geq 0$ be an integer. For $a>0$ and $T\geq2$, define the second moment of the $k$-derivative of $\zeta'/\zeta$ as
$$
I_k(a,T)=\int_1^T \left|\left(\dfrac{\zeta'}{\zeta}\right)^{\!\!\!(k)}\!\!\!\left( \frac{1}{2}+\frac{a}{\log T}+it\right) \right|^2  \!\!\mathrm{d}t.
$$
With this notation, we have $I_0(a,T)=I(a,T)$.

\begin{theorem} \label{Prin_theorem}
Assume RH and let $k\geq 0$ be an integer. The following statements are equivalent:
\begin{enumerate}
	\smallskip
	\item[\textup{(I)}] $\displaystyle \int_{b}^{b+\ell}\!\! \! F(\alpha,T) \,\d\alpha  \sim  \ell, \,\,\,\, \text{as } T\to \infty \, \text{ for any fixed} \  b\ge 1 \ {\rm and} \ \ell > 0$;
	\smallskip
	\item[\textup{(II)}] \ $\displaystyle I_k(a,T) \sim  \Bigg(\!\dfrac{(2k+1)!}{(2a)^{2k+2}}-\sum_{m=1}^{2k+1}\dfrac{m\,(2k)!}{(2k+1-m)!}\dfrac{e^{-2a}}{(2a)^{m+1}}\!\Bigg)T(\log T)^{2k+2},\,\,\,\, \text{as } T\to \infty \,\text{ for any fixed } a>0$.
\end{enumerate}
\end{theorem}

\!\!\!\!\!\!\!Note that Theorem \ref{Prin_theorem} gives new equivalences for the pair correlation conjecture. When $k=0$, it recovers the equivalence for the asymptotic formula \eqref{22_32pm}. Moreover, our result shows the dependence of the asymptotic formulae for $I_k(a,T)$ for all values of $k\geq 0$.

\begin{corollary} Assume RH. Then, the asymptotic formula \textup{(II)} holds for some $k\geq 0$ if and only if it holds for all $k\geq 0$.
\end{corollary}

\!\!\!\!\!\!\!One can estimate the right order of magnitude for $I_k(a,T)$, as $T\to \infty$, for a fixed $a>0$. In fact, using Proposition \ref{GGM_Lemma} and the uniform estimate (see, for instance \cite{GG})
\begin{align} \label{17:56pm}
\int_{1}^{\beta}\!\!F(\alpha,T)\,\d\alpha \ll \beta,
\end{align} 
it follows that for fixed $k\geq 0$ and $a>0$, we have $I_k(a,T) \asymp_{k,a} T (\log T)^{2k+2}$.

\medskip

On the other hand, Farmer proved a relation between $I(a,T)$ and a certain discrete mean value of $\zeta'/\zeta$. For $k\geq 0$ an integer, define
$$
D_{k}(a,T)=\sum_{0<\gamma\le T}\left(\dfrac{\zeta'}{\zeta}\right)^{\!\!\!(2k)}\!\!\!\left( \frac{1}{2}+\frac{a}{\log T}+i\gamma\right).
$$
Then, Farmer \cite[Lemma 3b]{Fa1} established that, for a fixed $a>0$,
\begin{equation}  \label{23_04pm}
	D_0(a,T)=\dfrac{1}{2\pi}\, I_0 \Big(\dfrac{a}{2},T\Big)+ O\big(T^\varepsilon\big), \,\,\,\,\text{for} \,\,T\geq 2\,\,\, \mbox{and}\,\, \varepsilon>0\,\, \mbox{sufficiently small}.
\end{equation} 
In particular, using \eqref{22_32pm} we obtain that the pair correlation conjecture is equivalent to \begin{equation*} 
	D_0(a,T) \sim  \left( \frac{1-e^{-a}}{2\pi a^2} \right) T \log^2 T, \,\,\,\, \text{as } T\to \infty \,\text{ for any fixed } a>0.
\end{equation*}
Extending \eqref{23_04pm} for $D_k(a,T)$ and using Theorem \ref{Prin_theorem} we arrive at the following corollary.
\begin{corollary} \label{17_05pm}
	Assume RH and let $k\geq 0$ be an integer. The following statements are equivalent:
	\begin{enumerate}
		\smallskip
		\item[\textup{(I)}] $\displaystyle \int_{b}^{b+\ell}\!\!\!F(\alpha,T) \,\d\alpha  \sim  \ell, \,\,\,\, \text{as } T\to \infty \,\text{ for any fixed} \  b\ge 1 \ {\rm and} \ \ell > 0$;
		\smallskip
		\item[\textup{(II)}] \ $\displaystyle D_k(a,T) \sim  \dfrac{1}{2\pi}\Bigg(\!\dfrac{(2k+1)!}{a^{2k+2}}-\sum_{m=1}^{2k+1}\dfrac{m\,(2k)!}{(2k+1-m)!}\dfrac{e^{-a}}{a^{m+1}}\!\Bigg)T(\log T)^{2k+2},\,\,\text{as } T\to \infty \,\text{ for any fixed } a>0$.
	\end{enumerate}
\end{corollary}
\vspace{0cm}

\subsection{Related results} We would like to point out that for some objects related to the Riemann zeta-function, there are results where one relates the asymptotic formula of their second moments to suitably weighted integrals of $F(\alpha,T)$. For instance, Goldston \cite[Theorem 1]{G2} showed, under RH, that
\begin{equation*} 
	\int_0^T |S(t)|^2 \mathrm{d}t = \frac{T}{2\pi^2}\log\log T + \frac{T}{2\pi^2}\left[\int_1^\infty \frac{F(\alpha,T)}{\alpha^{2}}\d \alpha+\gamma_0-\sum_{m=2}^\infty\sum_{p}\left(\dfrac{1}{m}-\dfrac{1}{m^2}\right)\dfrac{1}{p^m}\right] + o(T), \,\,\,\, \text{as } \,T\to \infty,
\end{equation*}
where $\pi S(t)$ denote the argument of the Riemann zeta-function at the point $\hh+it$, and $\gamma_0$ is Euler's constant. Recently, this has been extended to the iterates of the function $S(t)$ (see \cite[Theorem 1]{ChiOs}). Note that, assuming \eqref{22_09pm}, by integration by parts and  \eqref{17:56pm} we get 
\begin{equation*} 
	\int_0^T |S(t)|^2 \mathrm{d}t = \frac{T}{2\pi^2}\log\log T + \frac{T}{2\pi^2}\left[1+\gamma_0-\sum_{m=2}^\infty\sum_{p}\left(\dfrac{1}{m}-\dfrac{1}{m^2}\right)\dfrac{1}{p^m}\right] + o(T), \,\,\,\, \text{as } \,T\to \infty.
\end{equation*}
We refer the reader to Farmer \cite{Fa1, Fa2} for other results related to pair correlation and certain asymptotic formulae. 

\smallskip

\section{The representation formula for $I_k(a,T)$}

In this section, we establish a representation formula for the second moment of the $k$-derivative of $\zeta'/\zeta$, related to the function $F(\alpha,T)$. It can be seen as an extension of \cite[Theorem 1]{GGM}. The Poisson kernel plays an important role in our formula. For $b>0$, let $h_b:\R\to\R$ be the Poisson kernel defined as
\begin{align}\label{Def_Poisson_kernel_beta}
	h_b(x)=\dfrac{b}{b^2+x^2},
\end{align}
and let $\ell_b:\R\to\R$ be an auxiliary function\footnote{\,\,\,\,\,\,The function $\ell_b$ has previously been used to bound the real part of the derivative of $\zeta '/\zeta$ (see \cite[Theorem 3]{CGL}).} defined as
\begin{align} \label{19_13pm} 
	\ell_b(x)=\dfrac{b^2-x^2}{(b^2+x^2)^2}.
\end{align} 

\!\!\!\!\!\!\!The following technical lemma about the derivatives of $h_b$ and $\ell_b$ will be useful for us. 
\begin{lemma} \label{21_43pm}
Let $k\geq 0$ be an even integer. Then, for all $x\in\R$ we have
$$
\big|(h_b)^{(k)}(x)\big|\ll_k\dfrac{1}{b^{k-1}(b^2+x^2)}, \,\,\,\,\,\,\text{and} \,\,\,\,\, \,\,\,	\big|(\ell_b)^{(k)}(x)\big|\ll_k\dfrac{1}{b^{k}(b^2+x^2)}.
$$
\end{lemma}
\begin{proof}
Let us prove the first estimate for $b=1$. For any $k\geq 0$, it is easy to see by induction that
$$
(h_1)^{(k)}(x)=\dfrac{P(x)}{(1+x^2)^{2^k}},
$$
where $P$ is a polynomial of degree at most $2^{k+1}-k-2$. In particular, when $k=2m$ with $m\in\Z$ we have
$$
\big|(h_1)^{(2m)}(x)\big|\ll_m \dfrac{1}{(1+x^2)^{m+1}}.
$$
In the general case, since $h_b(x)=h_1(x/b)/b$, it follows that
$$
\big|(h_b)^{(2m)}(x)\big|=\dfrac{1}{b^{2m+1}}\Big|(h_1)^{(2m)}\Big(\frac{x}{b}\Big)\Big|\ll_m \dfrac{b}{(b^2+x^2)^{m+1}}\leq\dfrac{1}{b^{2m-1}(b^2+x^2)}. 
$$
We conclude the first estimate. The proof of the second estimate is similar.
\end{proof}

\smallskip

\begin{proposition} \label{GGM_Lemma} Assume RH and let $k\geq 1$ be a fixed integer. Then, for $0<a\ll 1$ and $T\geq 3$ we have
\begin{equation*} 	I_k(a,T)=\dfrac{(-1)^k\!\!\!}{2^{2k}\pi^{2k}}\,(\log T)^{2k+1}\!\!\!\!\!\!\sum_{0<\gamma,\gamma'\leq T}\!\big(h_{a/\pi}\big)^{\!(2k)}\bigg((\gamma-\gamma')\dfrac{\log T}{2\pi}\bigg)w(\gamma-\gamma') + O\bigg(\dfrac{T(\log T)^{2k+1}}{a^{2k-1}}+\dfrac{(\log T)^{2k+4}}{a^{2k+2}}\bigg),
	\end{equation*}
where $h_{a/\pi}$ is defined in \eqref{Def_Poisson_kernel_beta} and $w(u)=4/(4+u^2)$. In particular, for a fixed $a>0$, 
\begin{align} \label{1_58pm}
I_k(a,T)=\Bigg(\int_{0}^1\alpha^{2k+1}e^{-2a\alpha} \, \mathrm{d}\alpha +\int_{1}^\infty\alpha^{2k}e^{-2a\alpha} \, F(\alpha,T) \, \mathrm{d}\alpha + o(1)\Bigg)T(\log T)^{2k+2}, \,\,\,\, \text{as} \,\,\, T\to\infty.
\end{align}
\end{proposition} 
\begin{proof} We start obtaining a bound for $\big({\zeta'}/{\zeta}\big)^{\!(k)}$. Let $s=\sigma+it$, with $\hh<\sigma\leq \frac{3}{2}$ and $t\geq 2$. From the partial fraction decomposition for $\zeta'/\zeta$ \cite[Eq. 2.12.7]{Tit}
\begin{align}  \label{2_4722pm}
\dfrac{\zeta'}{\zeta}(s)=B-\dfrac{1}{s-1}+\dfrac{1}{2}\log\pi -\dfrac{1}{2}\dfrac{\Gamma'}{\Gamma}\bigg(\dfrac{s}{2}+1\bigg)+\sum_{\rho}\bigg(\dfrac{1}{s-\rho}+\dfrac{1}{\rho}\bigg),
\end{align}
where the sum runs over the non-trivial zeros $\rho=\hh+i\gamma$ of $\zeta(s)$ and $B=-\re\sum_\rho\rho^{-1}$. Taking $k$ derivatives in \eqref{2_4722pm} and using the estimate\footnote{\,\,\,\,\,\,It can be proved as the proof of Stirling's formula, but starting after taking $k$ derivatives in \cite[Eq. (34) in p. 202]{ahlfors}.}
$$
\left(\dfrac{\Gamma'}{\Gamma}\right)^{\!\!\!(k)}\!\!\!\!\!(w)=O\bigg(\dfrac{1}{|w|^k}\bigg), \,\,\,\,\,\mbox{for} \,\,\,\,\,\re{w}\geq \sigma_0>0,
$$
it follows that
\begin{align} \label{2_47pm}
\left(	\dfrac{\zeta'}{\zeta}\right)^{\!\!\!(k)}\!\!\!\!\!(s)=(-1)^k\,k!\,\sum_{\rho}\dfrac{1}{(s-\rho)^{k+1} }+O\bigg(\dfrac{1}{|t|^k}\bigg).
\end{align}
Since $\sum_{|\gamma-t|\leq 1}1=O(\log t)$, we have
\begin{align*}
\Bigg|\sum_{\gamma>t+1}\dfrac{1}{(s-\rho)^{k+1}}\Bigg|\leq \sum_{n\geq 1}\Bigg\{\!\sum_{t+n<\gamma\leq t+n+1}\dfrac{1}{|t-\gamma|^{k+1}}\!\Bigg\}\leq \sum_{n\geq 1}\Bigg\{\!\sum_{t+n<\gamma\leq t+n+1}\dfrac{1}{n^{k+1}}\!\Bigg\}\ll \sum_{n\geq 1}\dfrac{\log(t+n)}{n^{k+1}}\ll \log t.
\end{align*}
Similarly, we can prove the same estimate when the sum runs over $\gamma<t-1$. 
Therefore, in \eqref{2_47pm} we obtain,\footnote{\,\,\,\,\,\,The estimate \eqref{2_48pm} also holds when $k=0$.} for $\hh<\sigma\leq \frac{3}{2}$ and $t\geq 2$, 
\begin{align} \label{2_48pm}
\Bigg|\left(	\dfrac{\zeta'}{\zeta}\right)^{\!\!\!(k)}\!\!\!\!\!(\sigma+it)\Bigg|\ll \dfrac{\log t}{(\sigma-\hh)^{k+1}}.
\end{align}
Now, let us prove Proposition \ref{GGM_Lemma}. Using the elementary identity $
|w|^2= 2(\re\{w\})^2-\re\{w^2\}$ for all $w\in\C$, we write
\begin{align} \label{17_52pm} 
	\int_1^T \left|\left(	\dfrac{\zeta'}{\zeta}\right)^{\!\!\!(k)}\!\!\!\!\!(\sigma+it) \right|^2\!\!  \mathrm{d}t = 2\int_1^T\!\!\bigg(\re\left(	\dfrac{\zeta'}{\zeta}\right)^{\!\!\!(k)}\!\!\!\!\!(\sigma+it)\bigg)^{\!\!2}  \mathrm{d}t -\re\int_1^T\!\! \bigg(\!\left(	\dfrac{\zeta'}{\zeta}\right)^{\!\!\!(k)}\!\!\!\!\!(\sigma+it)\!\bigg)^{\!\!2} \mathrm{d}t .
\end{align} 
We estimate the second integral on the right-hand side of \eqref{17_52pm} by pulling the contour to the right, up to the line $\re{s}=\frac{3}{2}$ (see \cite[p. 111]{GGM}). In fact, to estimate the vertical edge at $\re{s}=\frac{3}{2}$ we use the representation as a Dirichlet series of $(\zeta'/\zeta)^{\!(k)}(s)$, and for the upper horizontal edge we use the estimate \eqref{2_48pm}. Therefore, in \eqref{17_52pm} we get 
\begin{align} \label{17_54pm} 
	\int_1^T \left|\left(	\dfrac{\zeta'}{\zeta}\right)^{\!\!\!(k)}\!\!\!\!\!(\sigma+it) \right|^2 \!\! \mathrm{d}t =2\int_1^T \!\!\bigg(\re\left(	\dfrac{\zeta'}{\zeta}\right)^{\!\!\!(k)}\!\!\!\!\!(\sigma+it)\bigg)^{\!\!2} \mathrm{d}t  +O\Bigg(\dfrac{\log ^2T}{(\sigma-\hh)^{2k+1}}\Bigg).
\end{align} 
On the other hand, note that
\begin{align*}
	(-1)^k\,k!\,\,&\re\biggl\{\sum_{\rho}\dfrac{1}{(s-\rho)^{k+1}}\biggr\} \\
	& = \sum_{\gamma}\re\biggl\{(-i)^k\dfrac{d^k}{dx^k}\Bigg(\dfrac{1}{(\sigma-\hh)+ix}\Bigg)\Biggr\}\Bigg|_{x=t-\gamma} \\
	& =\sum_{\gamma}\re\biggl\{(-i)^k\dfrac{d^k}{dx^k}\Bigg(\dfrac{\sigma-\hh}{(\sigma-\hh)^2+x^2}\Bigg)+(-i)^{k+1}\dfrac{d^k}{dx^k}\Bigg(\dfrac{x}{(\sigma-\hh)^2+x^2}\Bigg)\Biggr\}\Bigg|_{x=t-\gamma} \\
	& = \sum_{\gamma}\Bigg\{\re\big\{(-i)^k\big\}\dfrac{d^k}{dx^k}\Bigg(\dfrac{\sigma-\hh}{(\sigma-\hh)^2+x^2}\Bigg)+\re\big\{(-i)^{k+1}\big\}\dfrac{d^{k-1}}{dx^{k-1}}\Bigg(\dfrac{(\sigma-\hh)^2-x^2}{((\sigma-\hh)^2+x^2)^2}\Bigg)\Bigg\}\Bigg|_{x=t-\gamma}.
\end{align*}
Therefore, taking the real part of \eqref{2_47pm} we arrive at
\begin{align}  \label{2_472pm}
	\re\left(	\dfrac{\zeta'}{\zeta}\right)^{\!\!\!(k)}\!\!\!\!\!(\sigma+it)+O\bigg(\dfrac{1}{|t|^k}\bigg)=\sum_{\gamma}f_{k,\sigma}(t-\gamma),
\end{align}
where
$
f_{k,\sigma}(x)=\re\{(-i)^k\}\,(h_{\sigma-1/2})^{(k)}(x)+\re\{(-i)^{k+1}\}\,(\ell_{\sigma-1/2})^{(k-1)}(x),
$
and the functions $h_{\sigma-1/2}$ and $\ell_{\sigma-1/2}$ are defined in \eqref{Def_Poisson_kernel_beta} and \eqref{19_13pm} respectively. Using the Fourier transforms\footnote{\,\,\,\,\,\,For a function $f \in L^1(\mathbb{R})$, we define its Fourier transform as
	$\widehat{f}(y) = \int_{-\infty}^\infty  e^{-2\pi i y x} \, f(x)\,\mathrm{d}x
	$, and the convolution of $f$ and $g$ is defined as $(f\ast g)(y)=\int_{-\infty}^\infty  f(x) \,g(y-x)\,\mathrm{d}x$.} 
$$
\widehat{\,h_b\,}(y)= \pi e^{-2\pi b|y|} \,\,\,\,\,\,  \mbox{and} \,\,\,\,\,\, \widehat{\,\ell_b\,}(y)= 2\pi^2|y|e^{-2\pi b|y|},
$$
the Fourier transform of $f_{k,\sigma}$  is given by
\begin{align} \label{15_00}
\widehat{f_{k,\sigma}}(y)=\Big(\big(\re\big\{i^k\big\}\big)^{\!2}y^k+\big(\re\big\{i^{k+1}\big\}\big)^{\!2}y^{k-1}|y|\Big)(-1)^k2^k\pi^{k+1}e^{-2\pi(\sigma-1/2)|y|}.
\end{align} 
Now, we square \eqref{2_472pm}, integrate from $1$ to $T$, and use \eqref{2_48pm} to get
\begin{align}  \label{23_57}
\int_1^T\!\left(\!\re\left(	\dfrac{\zeta'}{\zeta}\right)^{\!\!\!(k)}\!\!\!\!\!(\sigma+it)\!\right)^{\!\!2}\!\!\mathrm{d}t + O\Bigg(\dfrac{\log ^2T}{(\sigma-\hh)^{k+1}}\Bigg)= \int_{1}^T\!\left(\sum_{\gamma}f_{k,\sigma}(t-\gamma)\right)^{\!\!2}\!\!\mathrm{d}t.
\end{align} 
We proceed to analyze the right-hand side of  \eqref{23_57}. From Lemma \ref{21_43pm}, it follows that\footnote{\,\,\,\,\,\,We highlight that depending on the parity of $k$, only one of the terms of $f_{k,\sigma}$ appears.}
$$|f_{k,\sigma}(x)|\ll \dfrac{h_{\sigma-1/2}(x)}{(\sigma-1/2)^k},$$ and using Montgomery's argument \cite{M} we can restrict the inner sum over the zeros of $\zeta(s)$ such that $0<\gamma\leq T$ and extend the integral to all $t\in \R$, with a final error at most $\ll  (\sigma-1/2)^{-2k}\log^3T + (\sigma-1/2)^{-2k-2}{\log^2T}$ (see \cite[p. 113]{GGM}). Therefore, from \eqref{15_00} and using the fact that $f_{k,\sigma}$ is even, 
\begin{align*} 
	\int_{-\infty}^{\infty}\left(\sum_{0<\gamma\leq T}f_{k,\sigma}(t-\gamma)\right)^{\!\!2} \!\! \mathrm{d}t & = \sum_{0<\gamma,\gamma'\leq T}\big(f_{k,\sigma}\ast f_{k,\sigma}\big)(\gamma-\gamma')  = \sum_{0<\gamma,\gamma'\leq T}\widehat{\big(\widehat{ \!\!\!\!\!f_{k,\sigma}\,\,\,\,\,}\big)^{2}}(\gamma-\gamma') \\
	& =\pi(-1)^k\!\sum_{0<\gamma,\gamma'\leq T}\big(h_{2\sigma-1}\big)^{\!(2k)}{(\gamma-\gamma')}.
\end{align*} 
We want to add the weight $w(\gamma-\gamma')$ to the last sum. In fact, Lemma \ref{21_43pm} gives the bound
\begin{align*}
\Bigg|\sum_{0<\gamma,\gamma'\leq T}\big(h_{2\sigma-1}\big)^{\!(2k)}(\gamma-\gamma')\big(1-w(\gamma-\gamma')\big)\bigg|&\ll\dfrac{1}{(2\sigma-1)^{2k-1}}\sum_{0<\gamma,\gamma'\leq T}\dfrac{4}{4+(\gamma-\gamma')^2}\\
& \ll \dfrac{T\log T\,F(0,T)}{(2\sigma-1)^{2k-1}}\ll \dfrac{T\log^2T}{(2\sigma-1)^{2k-1}},
\end{align*}where in the last estimate we have used \eqref{F formula}. Thus,
\begin{align*}  
	\int_{1}^T\!\left(\sum_{\gamma}f_{k,\sigma}(t-\gamma)\right)^{\!\!\!2}\!\mathrm{d}t & =\pi(-1)^k\!\!\sum_{0<\gamma,\gamma'\leq T}\!\big(h_{2\sigma-1}\big)^{\!(2k)}{(\gamma-\gamma')}\,w(\gamma-\gamma')  \\
	& \,\,\,\,\,\,+ O\bigg(\dfrac{T\log^2T}{(2\sigma-1)^{2k-1}}+\dfrac{\log^3T}{(2\sigma-1)^{2k}}+ \dfrac{\log^2T}{(2\sigma-1)^{2k+2}}\bigg). 
\end{align*} 
Now, considering that $\sigma=\hh+\frac{a}{\log T}$  \,for $0<a\ll 1$ and using the fact that $h_{2\sigma-1}(x)=h_{a/\pi}(x\log T/2\pi)\log T/2\pi$ for $x\in\R$, we obtain in \eqref{23_57} 
\begin{align*} 
	\int_1^T\!\left(\!\re\left(	\dfrac{\zeta'}{\zeta}\right)^{\!\!\!(k)}\!\!\!\!\!(\sigma+it)\!\right)^{\!\!2}\!\!\mathrm{d}t &=\dfrac{(-1)^k}{2^{2k+1}\pi^{2k}}\,(\log T)^{2k+1}\!\!\sum_{0<\gamma,\gamma'\leq T}\!\big(h_{a/\pi}\big)^{\!(2k)}\bigg((\gamma-\gamma')\dfrac{\log T}{2\pi}\bigg)w(\gamma-\gamma') \\
	& \,\,\,\,\,\,\,\,+ O\bigg(\dfrac{T(\log T)^{2k+1}}{a^{2k-1}}+\dfrac{(\log T)^{2k+4}}{a^{2k+2}}\bigg). 
\end{align*} 
Inserting it in \eqref{17_54pm}  we conclude that
\begin{align}  \label{00_55pm}
	\begin{split} 
		\int_1^T \left|\left(\dfrac{\zeta'}{\zeta}\right)^{\!\!\!(k)}\!\!\!\left( \frac{1}{2}+\frac{a}{\log T}+it\right)\right|^2 \! \!\mathrm{d}t  &=\dfrac{(-1)^k}{2^{2k}\pi^{2k}}\,\,(\log T)^{2k+1}\!\!\sum_{0<\gamma,\gamma'\leq T}\!\big(h_{a/\pi}\big)^{\!(2k)}\bigg((\gamma-\gamma')\dfrac{\log T}{2\pi}\bigg)w(\gamma-\gamma') \\
		& \,\,\,\,\,\,\,\,+ O\bigg(\dfrac{T(\log T)^{2k+1}}{a^{2k-1}}+\dfrac{(\log T)^{2k+4}}{a^{2k+2}}\bigg).
	\end{split} 
\end{align}
From Fourier inversion, it is known that for any function $R \in L^1(\mathbb{R})$ such that $\widehat{R} \in L^1(\mathbb{R})$ we have the formula (see \cite[Eq. (3)]{M})
\begin{equation*} 
	\sum_{0<\gamma,\gamma'\le T} R\!\left((\gamma-\gamma') \frac{\log T}{2\pi} \right) w(\gamma-\gamma') = \frac{T\log T}{2\pi} \int_{-\infty}^\infty\widehat{R}(\alpha) \, F(\alpha,T) \, \mathrm{d}\alpha.
\end{equation*}
Applying this formula to the function $(h_{a/\pi})^{(2k)}$ and using the fact that $\widehat{(h_b)^{(2k)}}(y)=(-1)^k2^{2k}\pi^{2k+1}y^{2k}e^{-2\pi b|y|}$, we get in \eqref{00_55pm} that, for a fixed $a>0$,
\begin{align} \label{1_38am}  
		\int_1^T \left|\left(\dfrac{\zeta'}{\zeta}\right)^{\!\!\!(k)}\!\!\!\left( \frac{1}{2}+\frac{a}{\log T}+it\right)\right|^2 \! \mathrm{d}t    &=
\dfrac{T(\log T)^{2k+2}}{2}\int_{-\infty}^\infty\!\alpha^{2k}e^{-2a|\alpha|} \, F(\alpha,T) \, \mathrm{d}\alpha + O\big(T(\log T)^{2k+1}\big). 
\end{align}
Refining the original work of Montgomery \cite{M}, Goldston and Montgomery \cite[Lemma 8]{GM} proved that, under RH, 
\begin{equation}\label{F formula}
	F(\alpha,T) = \big(T^{-2|\alpha|}\log T + |\alpha| \big)(1+o(1)) , \quad \text{as } T\to \infty, 
\end{equation} 
uniformly for $0\le |\alpha| \le 1$.  Using \eqref{F formula} and the fact that $F(\alpha,T)=F(-\alpha,T)$ for all $\alpha\in\R$, we have 
$$
\int_{-\infty}^\infty\!\alpha^{2k}e^{-2a|\alpha|} \, F(\alpha,T) \, \mathrm{d}\alpha = 2\int_{0}^1\!\alpha^{2k+1}e^{-2a\alpha} \, \mathrm{d}\alpha +2\int_{1}^\infty\!\alpha^{2k}e^{-2a\alpha}F(\alpha,T)  \, \mathrm{d}\alpha + o(1).
$$
Inserting this in \eqref{1_38am} we arrive at \eqref{1_58pm}. 
\end{proof}

\section{A Tauberian lemma and the Proof of Theorem \ref{Prin_theorem}}

\subsection{A Tauberian lemma} The following lemma can be seen as a generalization\footnote{\,\,\,\,\,\,See \cite{Bayulot} for another extension of \cite[Lemma 2]{GGM} depending of certain measures.} of  \cite[Lemma 2]{GGM}, where the case $G\equiv 1$  was considered. The proof uses Karamata's method and some examples of these Tauberian lemmas are given in \cite[Section 7.12]{Tit}.

\begin{lemma}\label{16_10pm} Let $f(\alpha, T)\geq 0$ be a function such that the function $\alpha\mapsto f(\alpha, T)$ is continuous for each $T\geq 2$ fixed, and for $\beta>0$ and $T\geq 2$,	\begin{align}  \label{24_36pm}
	\int_{0}^{\beta}\!f(\alpha,T)\, \mathrm{d}\alpha  \ll\beta+1.
\end{align} 
Let $G$ be a polynomial such that $G(\alpha)>0$ for $\alpha \in[0,\infty)$. The following statements are equivalent:
\begin{enumerate}
	\smallskip
	\item[\textup{(A)}] $\displaystyle\int_{0}^\infty \!f(\alpha,T)\,G(\alpha)\,e^{-b\alpha}\, \mathrm{d}\alpha \sim \int_{0}^\infty \!G(\alpha)\,e^{-b\alpha}\, \mathrm{d}\alpha, \!\quad \text{as } T\to \infty \text{ for any fixed } b>0$.
	
	\smallskip
	
	\item[\textup{(B)}] $\displaystyle\dfrac{1}{d-c}\int_{c}^d \!f(\alpha,T)\, \mathrm{d}\alpha \sim 1,\!\quad \text{as } T\to \infty \text{ for any fixed } 0\leq c<d.$
	\smallskip
\end{enumerate}
\end{lemma}
\begin{proof}
Let us start assuming (A). Let $0\leq c<d$ be fixed, and define the function $h:[0,1]\to \R$ by
$$ 
	h(u)= \left\{ \begin{array}{lcc}
	0, &  \!\!\! \text{if}  &  \!\!\!\!\! 0\leq u <e^{-d} 
	\\ \dfrac{1}{u\,G(-\log u)}, & \!\!\!  \text{if}  &\! e^{-d} \leq u \leq e^{-c}
	\\ 0, & \!\!\! \text{if}  &\!\!\!\!\!   \!e^{-c} < u \leq 1. 
\end{array}
\right.$$
By the Weierstrass approximation theorem, for any $\varepsilon>0$ sufficiently small we can construct a polynomial $P(u)=\sum_{n=0}^{N}a_nu^n$ (depending on $\varepsilon$) such that 
\begin{align} \label{0_53pm}
h(u)\leq P(u)\,\,\,\text{for all}\,\, u\in [0,1
], \,\,\, \,\,\,\text{and} \,\,\,\,\,\, \int_{0}^{1}\!(P(u)-h(u))^2 \, \mathrm{d}u =O(\varepsilon).
\end{align}
Defining the function $Q(\alpha)=e^{-\alpha}P\big(e^{-\alpha}\big)$, it follows that\footnote{\,\,\,\,\,\,Here $\chi_{[c,d]}(\alpha)$ denotes the characteristic function of the interval $[c,d]$.} $$\frac{\chi_{[c,d]}(\alpha)}{G(\alpha)}\leq Q(\alpha)$$ for all $\alpha\geq 0$. Recalling that $G(\alpha)>0$ we have
\begin{align*}
\int_{c}^{d}\!f(\alpha,T)\, \mathrm{d}\alpha  \leq \int_{0}^{\infty}\!f(\alpha,T)\,G(\alpha)\,{Q(\alpha )}\, \mathrm{d}\alpha & = \int_{0}^{\infty}\!f(\alpha,T)\,G(\alpha)\,\sum_{n=0}^{N}a_ne^{-(n+1)\alpha}\, \mathrm{d}\alpha \\
& = \sum_{n=0}^{N}a_n\int_{0}^{\infty}\!f(\alpha,T)\,G(\alpha)\,e^{-(n+1)\alpha}\,\mathrm{d}\alpha.
\end{align*}
Taking $\limsup$ as $T\to\infty$ and using (A) we arrive at
\begin{align} \label{1_09am}
	\limsup_{T\to\infty}\int_{c}^{d}\! f(\alpha,T)\, \mathrm{d}\alpha \leq  \sum_{n=0}^{N}a_n\int_{0}^{\infty}\!G(\alpha)\,e^{-(n+1)\alpha}\,\mathrm{d}\alpha
\end{align}
By a change of variables, the definition of $h$, the Cauchy-Schwarz inequality and \eqref{0_53pm}, one can see that
\begin{align*}
\sum_{n=0}^{N}a_n\int_{0}^{\infty}\!G(\alpha)\,e^{-(n+1)\alpha}\,\mathrm{d}\alpha& =\int_{0}^{1}\!G(-\log u)\,P(u)\,\mathrm{d}u\\
& = \int_{0}^{1}\!G(-\log u)\,h(u)\,\mathrm{d}u +\int_{0}^{1}\!G(-\log u)\,(P(u)-h(u))\,\mathrm{d}u \\
& = \int_{e^{-d}}^{e^{-c}}\dfrac{1}{u}\,\mathrm{d}u +O\left(\Bigg(\int_{0}^{1}\!G^2(-\log u)\,\mathrm{d}u\Bigg)^{\!\!1/2}\Bigg(\int_{0}^{1}\!(P(u)-h(u))^2\,\mathrm{d}u\Bigg)^{\!\!1/2}\right) \\
& = d-c + O(\varepsilon^{1/2}).
\end{align*} 
Letting $\varepsilon\to 0$ and combining this with \eqref{1_09am}, we conclude that
\begin{align*} 
	\limsup_{T\to\infty}\int_{c}^{d}\!f(\alpha,T)\, \mathrm{d}\alpha \leq d-c.
\end{align*}
Similarly, we can proceed to prove that 
\begin{align*} 
	d-c\leq \liminf_{T\to\infty}\int_{c}^{d}f(\alpha,T)\, \mathrm{d}\alpha.
\end{align*}
Therefore we obtain (B). Let us prove that (B) implies (A). Using integration by parts and \eqref{24_36pm}, we see that
\begin{align}
	\displaystyle\int_{0}^\infty \!f(\alpha,T)\,G(\alpha)\,e^{-b\alpha}\, \mathrm{d}\alpha  =  - 	\displaystyle\int_{0}^\infty\!\Bigg(\int_{0}^\alpha \!f(\beta,T)\, \mathrm{d}\beta\Bigg) \big(G(\alpha)\,e^{-b\alpha}\big)'\, \mathrm{d}\alpha.
\end{align}
Finally, using (B), the dominated convergence theorem, and integration by parts one more time, we conclude.
\end{proof}

\subsection{Proof of Theorem \ref{Prin_theorem}}
Since the case $k=0$ was considered in the work of Goldston, Gonek and Montgomery (see \cite[Theorem 3]{GGM}), assume $k\geq 1$. Using the identity\footnote{\,\,\,\,\,\,See \cite[Eq. 3.351-1 and 3.351-2]{GR}.}
$$
\int_{0}^1\!\alpha^{2k+1}e^{-2a\alpha} \, \mathrm{d}\alpha +\int_{1}^\infty\!\alpha^{2k}e^{-2a\alpha}  \, \mathrm{d}\alpha = \dfrac{(2k+1)!}{(2a)^{2k+2}}-\sum_{m=1}^{2k+1}\dfrac{m\,(2k)!}{(2k+1-m)!}\dfrac{e^{-2a}}{(2a)^{m+1}}, \,\,\,\,\,\,\mbox{for any} \,\,a>0,
$$
and \eqref{1_58pm} we have that (II) is equivalent to
$$
\int_{1}^\infty\!\alpha^{2k}e^{-2a\alpha} \, F(\alpha,T) \, \mathrm{d}\alpha \sim \int_{1}^\infty\!\alpha^{2k}e^{-2a\alpha} \, \mathrm{d}\alpha.
$$
A translation gives that (II) is equivalent to
$$
\int_{0}^\infty\!(\alpha+1)^{2k}e^{-2a\alpha} \, F(\alpha+1,T) \, \mathrm{d}\alpha \sim \int_{0}^\infty\!(\alpha+1)^{2k}e^{-2a\alpha} \, \mathrm{d}\alpha.
$$
Using Lemma \ref{16_10pm} with the function $f(\alpha,T)=F(\alpha+1,T)$, $G(\alpha)=(\alpha+1)^{2k}$, and $b=2a$ we conclude the proof. We remark that the additional constraint \eqref{24_36pm} follows from \eqref{17:56pm}.

\medskip

\section{Proof of Corollary \ref{17_05pm}} Assume RH. From \cite[p. 340]{Tit}, for each $n\in \N$  there is $T_n\in (n,n+1)$ such that for $-1\leq\sigma\leq 2$,
\begin{align} \label{19_33}
\bigg|\dfrac{\zeta'}{\zeta}(\sigma+iT_n)\bigg|\ll (\log T_n)^2.
\end{align} 
Now, let $k\geq 1$ be an integer, $0<a\ll 1$ and $T\geq 4$ $, T\notin\N$. Choose $n\in \N$ such that $T, T_n \in (n,n+1)$ and $T_n$ satisfies \eqref{19_33}. Note that $\log T_n \asymp \log T$. Using integration by parts $k$ times and the bound \eqref{2_48pm}, we have
\begin{align*} 
\int_1^{T_n} \left|\left(\dfrac{\zeta'}{\zeta}\right)^{\!\!\!(k)}\!\!\!\left( \frac{1}{2}+\dfrac{a}{\log T}+it\right) \right|^2  \!\!\mathrm{d}t & = \int_1^{T_n} \left(\dfrac{\zeta'}{\zeta}\right)^{\!\!\!(k)}\!\!\!\left( \frac{1}{2}+\dfrac{a}{\log T}+it\right)\,\left(\dfrac{\zeta'}{\zeta}\right)^{\!\!\!(k)}\!\!\left( \frac{1}{2}+\dfrac{a}{\log T}-it\right)  \mathrm{d}t   \\
& = \dfrac{1}{i}\int_{\hh-\frac{a}{\log T}+i}^{\hh-\frac{a}{\log T}+iT_n} \left(\dfrac{\zeta'}{\zeta}\right)^{\!\!\!(k)}\!\!\!\left(s + \dfrac{2a}{\log T}\right)\,\left(\dfrac{\zeta'}{\zeta}\right)^{\!\!\!(k)}\!\!\!\!\left(1-s\right)  \mathrm{d}s \\
& = \dfrac{1}{i}\int_{\hh-\frac{a}{\log T}+i}^{\hh-\frac{a}{\log T}+iT_n} \left(\dfrac{\zeta'}{\zeta}\right)^{\!\!\!(2k)}\!\!\!\left(s + \dfrac{2a}{\log T}\right)\dfrac{\zeta'}{\zeta}\!\left(1-s\right)  \mathrm{d}s + O\bigg(\dfrac{(\log T)^{2k+3}}{a^{2k+1}}\bigg).  
\end{align*}
We use the residue theorem on the rectangle with vertices $\hh-\frac{a}{\log T}+i, 2+i, 2+iT_n$ and $\hh-\frac{a}{\log T}+iT_n$ (since RH holds, the function $(\zeta'/\zeta)^{(2k)}\big(s+\frac{2a}{\log T}\big)$ is analytic in this rectangle) and the bounds \eqref{2_48pm} and \eqref{19_33} to deduce that
\begin{align*}
	\int_1^{T_n} \left|\left(\dfrac{\zeta'}{\zeta}\right)^{\!\!\!(k)}\!\!\!\left( \frac{1}{2}+\dfrac{a}{\log T}+it\right) \right|^2  \!\!\mathrm{d}t
	& = 2\pi\sum_{0<\gamma<T_n}\left(\dfrac{\zeta'}{\zeta}\right)^{\!\!\!(2k)}\!\!\!\left(\rho+ \dfrac{2a}{\log T}\right)\\
	& \,\,\,\,\,\,\, +\dfrac{1}{i}\int_{2+i}^{2+iT_n} \left(\dfrac{\zeta'}{\zeta}\right)^{\!\!\!(2k)}\!\!\!\left(s + \dfrac{2a}{\log T}\right)\,\dfrac{\zeta'}{\zeta}\!\left(1-s\right)  \mathrm{d}s + O\bigg(\dfrac{(\log T)^{2k+4}}{a^{2k+1}}\bigg).
\end{align*}
It is known that $\zeta(s)$ satisfies the functional equation $\zeta(s)=\chi(s)\zeta(1-s)$, where
$$
\chi(s)=\dfrac{\pi^{s-\frac{1}{2}} \Gamma(\hh-\frac{s}{2} )}{\Gamma(\frac{s}{2})}.
$$
Then, we write
\begin{align} \label{20_26}
\dfrac{1}{i}\int_{2+i}^{2+iT_n} \left(\dfrac{\zeta'}{\zeta}\right)^{\!\!\!(2k)}\!\!\!\left(s + \dfrac{2a}{\log T}\right)\,\dfrac{\zeta'}{\zeta}\!\left(1-s\right)  \mathrm{d}s = \dfrac{1}{i}\int_{2+i}^{2+iT_n} \left(\dfrac{\zeta'}{\zeta}\right)^{\!\!\!(2k)}\!\!\!\left(s + \dfrac{2a}{\log T}\right)\,\bigg(\dfrac{\chi'}{\chi}\!\left(s\right)-\dfrac{\zeta'}{\zeta}\!\left(s\right)\bigg) \mathrm{d}s.
\end{align} 
Using the estimate
$$
\dfrac{\chi'}{\chi}(\sigma+it)=-\log\bigg|\dfrac{t}{2\pi}\bigg| + O\bigg(\dfrac{1}{|t|}\bigg),  \,\,\, \text{for}\,\,\, |t|\geq 1 \,\,\,
\text{and} \,\,\,|\sigma|\ll 1,
$$
and the representation as a Dirichlet series of $(\zeta'/\zeta)^{(2k)}(s)$ in the right hand-side of \eqref{20_26}, we integrate term by term the right-hand side of \eqref{20_26} to obtain $O(\log T)$. Therefore, we arrive at
\begin{align*}
	\int_1^{T_n} \left|\left(\dfrac{\zeta'}{\zeta}\right)^{\!\!\!(k)}\!\!\!\left( \frac{1}{2}+\dfrac{a}{\log T}+it\right) \right|^2  \!\!\mathrm{d}t
	& = 2\pi\sum_{0<\gamma<T_n}\left(\dfrac{\zeta'}{\zeta}\right)^{\!\!\!(2k)}\!\!\!\left(\rho+ \dfrac{2a}{\log T}\right)+ O\bigg(\dfrac{(\log T)^{2k+4}}{a^{2k+1}}\bigg).
\end{align*}
We can replace $T_n$ by $T$ using \eqref{2_48pm} and $\sum_{|t-\gamma|\leq 1}1=O(\log t)$ with an error at most $\ll (\log T)^{2k+4}/a^{2k+2}$. Therefore, we conclude for $0<a\ll  1$ and $T$ sufficiently large, that
\begin{align*}
	I_k(a,T)= 2\pi D_k(2a,T)+ O\bigg(\dfrac{(\log T)^{2k+4}}{a^{2k+2}}\bigg).
\end{align*}
Finally, we use Theorem \ref{Prin_theorem} to conclude.

\section*{Acknowledgments}
\noindent A.C. was supported by Grant $275113$ of the Research Council of Norway. I would like to thank Oscar Quesada-Herrera and the referee of this paper for their valuable suggestions.


\begin{thebibliography}{9999}

\bibitem{ahlfors}
L. V. Ahlfors, 
\newblock {\it Complex analysis. An introduction to the theory of analytic functions of one complex variable.},
\newblock Third edition. International Series in Pure and Applied Mathematics. McGraw-Hill Book Co., New York, 1978.






\bibitem{Bayulot} 
S.~Baluyot, 
\newblock On the pair correlation conjecture and the alternative hypothesis, 
\newblock J.~Number Theory 169 (2016), 183--226.

\bibitem{CCLM}
E.~Carneiro, V.~Chandee, F.~Littmann, and M.~B.~Milinovich, 
\newblock Hilbert spaces and the pair correlation of zeros of the Riemann zeta-function,
\newblock J.~Reine Angew. Math. 725 (2017), 143--182.




\bibitem{CCCM}
E.~Carneiro, V.~Chandee, A.~Chirre, and M.~B.~Milinovich,
\newblock On Montgomery's pair correlation conjecture: A tale of three integrals,
\newblock preprint.


\bibitem{CGL}
A.~Chirre and F. Gon\c{c}alves,  
\newblock Bounding the log-derivative of the zeta-function,
\newblock to appear in Math. Z.

\bibitem{ChiOs}
A. Chirre and O. E. Quesada-Herrera,  
\newblock The second moment of $S_n(t)$ on the Riemann hypothesis,
\newblock to appear in Int. J. Number Theory.






\bibitem{Fa1}
D.~W.~Farmer, 
\newblock Long mollifiers of the Riemann zeta-function, 
\newblock Mathematika 40 (1993), no. 1, 71--87.

\bibitem{Fa2}
D.~W.~Farmer, 
\newblock Mean values of $\zeta'/\zeta$ and the Gaussian unitary ensemble hypothesis,
\newblock Internat. Math. Res. Notices (1995), no. 2, 71--82.

\bibitem{Ga}
P.~X.~Gallagher, 
\newblock Pair correlation of zeros of the zeta function, 
\newblock J. Reine Angew. Math. 362 (1985), 72--86.

\bibitem{GaMu}
P.~X.~Gallagher and J.~H.~Mueller, 
\newblock Primes and zeros in short intervals,
\newblock J. Reine Angew. Math. 303/304 (1978), 205--220.

\bibitem{G2} 
D.~A.~Goldston,
\newblock On the function $S(T)$ in the theory of the Riemann zeta-function,
\newblock J. Number Theory 27 (1987), no. 2, 149--177.

\bibitem{G} 
D.~A.~Goldston,
\newblock On the pair correlation conjecture for zeros of the Riemann zeta-function,
\newblock J. Reine Angew. Math. 385 (1988), 24--40.

\bibitem{GG} 
D.~A.~Goldston and S.~M.~Gonek,
\newblock A note on the number of primes in short intervals,
\newblock Proc. Amer. Math. Soc. 108 (1990), no. 3, 613--620.
 
\bibitem{GGM} 
D.~A.~Goldston, S.~M.~Gonek, and H.~L.~Montgomery,
\newblock Mean values of the logarithmic derivative of the Riemann zeta-function with applications to primes in short intervals, 
\newblock J.~Reine Angew.~Math. 537 (2001), 105--126.

\bibitem{GM} 
D.~A.~Goldston and H.~L.~Montgomery, 
\newblock Pair correlation of zeros and primes in short intervals,
\newblock Analytic number theory and Diophantine problems (Stillwater, OK, 1984), 183--203, Progr. Math., 70, Birkh\"auser Boston, Boston, MA, 1987.

\bibitem{GR} I. S. Gradshteyn and I. M. Ryzhik,
\newblock {\it Table of integrals, series, and products}, 
\newblock Translated from Russian. Translation edited and with a preface by Alan Jeffrey and Daniel Zwillinger. Seventh edition. Elsevier/Academic Press, Amsterdam (2007).



\bibitem{M}
H.~L.~Montgomery, 
\newblock The pair correlation of zeros of the zeta function,
\newblock Analytic number theory (Proc. Sympos. Pure Math., Vol. XXIV, St. Louis Univ., St. Louis, Mo., 1972), pp. 181--193. Amer. Math. Soc., Providence, R.I., 1973.




\bibitem{Tit}
E. C. Titchmarsh, 
\newblock {\it The theory of the Riemann zeta-function},
\newblock Second edition, Edited and with a preface by D. R. Heath-Brown, The Clarendon Press, Oxford University Press, New York, 1986.




\end{thebibliography}
\end{document}